\documentclass{amsart}
\usepackage{comment,xypic}

\newcommand{\N}{\mathbb{N}}
\newcommand{\Hom}{\operatorname{Hom}}
\newcommand{\soc}{\operatorname{soc}}

\newtheorem{Theorem}{Theorem}[section]
 
\newtheorem{Lemma}[Theorem]{Lemma}

\newtheorem{Corollary}[Theorem]{Corollary}

\theoremstyle{definition}

\author{Joel Kamnitzer}
\author{Chandrika Sadanand}
\title[Modules with 1-dimensional socle]{Modules with 1-dimensional socle and components of Lusztig quiver varieties in type A}
\date{\today}

\begin{document}
\maketitle

\section{Introduction}
For any simply-laced Kac-Moody Lie algebra $ \mathfrak{g} $, Lusztig \cite{L} has constructed canonical bases for its representations using the geometry of quiver varieties.  In particular, Lustzig considered the variety $Rep(w)_v $ of representations of the preprojective algebra $ \Lambda $ on a fixed vector space of dimension $ v $ and having dimension of socle bounded by $ w $.  The irreducible components of this variety index Lusztig's canonical basis for a particular weight space of a highest weight representation of $ \mathfrak{g} $.  The components of $ Rep(w)_v $ are also in natural bijection with the components of Nakajima's Lagrangian quiver varieties.  This is shown in the work of Saito \cite[section 4.6]{Sai}, who also studied a crystal structure on these components jointly with Kashiwara \cite{KS}.

Because the components of $ Rep(w)_v $ index the canonical basis, it would be interesting to descibe them in an explicit fashion using known combinatorics.  In certain special cases (including $ \mathfrak{g} = \mathfrak{sl}_n $), this has been done by Savage \cite{Sav}, using ad-hoc methods.  In a forthcoming paper \cite{BK}, Pierre Baumann and the first author will use module-theoretic means to give a uniform description of the components using the theory of MV polytopes \cite{K}.  In our description, a key role is played by certain $\Lambda$-modules with one dimensional socle.  

In the current paper, we focus on the case $ \mathfrak{g} = \mathfrak{sl}_n $.  Using elementary means, we classify $\Lambda$-modules with one dimensional socle and explain how these modules can be used to describe components of $ Rep(w)_v $.  Similar results (and more) will be formulated and proved for general $ \mathfrak{g} $ in \cite{BK}.

More specifically in section \ref{se:class}, we classify $\Lambda$-modules with one dimensional socle by showing that they are all isomorphic to certain Maya modules introduced by Savage \cite{Sav}.  These Maya modules are in bijection with subsets of $ \{1, \dots, n \} $ (other than $ \{1, \dots, i \} $).  Next, we compute the space of homomorphisms between two such modules, obtaining an explicit combinatorial formula.  We show that this formula is related to a truncated permutahedron, which is the MV polytope for this situation. 

In section \ref{se:tableaux}, we show how Maya modules can be used to describe the components of $ Rep(w)_v$.  We begin by computing the space of homomorphisms between certain Maya modules and modules associated to tableaux by Savage \cite{Sav}.  We use this to rephrase Savage's description of the components in a module-theoretic fashion.

\subsection{Acknowledgements}
We would like to thank Pierre Baumann for allowing us to use some ideas from our forthcoming joint work.  We would also thank Alistair Savage for his generous explanations.  We thank Bernard Leclerc and Bernhard Keller for helpful conversations.  Finally, we thank Lucy Zhang and the referee for their comments on this paper.  Part of this work was conducted by C.S. as an NSERC undergraduate summer research project.

\section{Background}

\subsection{Notation}
Let $ Q $ denote the root lattice of $ SL_n $.  So 
\begin{equation*}
 Q = \{ (x_1, \dots, x_n) \in \mathbb{Z}^n: \sum x_i = 0 \}.
\end{equation*}

For $ i = 1, \dots, n-1$, let $ \alpha_i = (\dots, 0, 1, -1, 0, \dots) $ denote the simple roots (the 1 is in the $i$th position).  Let $ Q_+ $ be the subset of $ Q$ given by non-negative sums of the $ \alpha_i $.  Let $ \omega_i = (1, \dots, 1, 0, \dots, 0) $ denote the fundamental weights (the first $ i $ entries are 1s).

If $ A $ and $ B $ are $ i $ element subsets of $ \{1, \dots, n \} $, then we define $$ A - B := 1_A - 1_B \in Q $$ where $1_A $ is the $n$-tuple which is 1 in positions indexed by numbers in $ A $ and 0 in the other positions.  We write $ A \ge B $ if $ A - B \in Q_+ $.

\subsection{The preprojective algebra}
Let $ \Omega$ be a simply-laced Dynkin quiver (that is a Dynkin diagram with orientation) with edge set $ \Omega $ and vertex set $ I$.  Let $ \Lambda $ denote the preprojective algebra of the quiver $ \Omega $.  By definition $ \Lambda $ is the quotient
\begin{equation*}
P(\Omega \oplus \overline{\Omega})/( \sum_{\tau \in \Omega} \tau \overline{\tau} - \overline{\tau} \tau)
\end{equation*}
of the path algebra of the doubled quiver $ \Omega \oplus \overline{\Omega} $ by the preprojective relation.

For this paper, we will work exclusively with the type $A_{n-1}$ quiver with the leftward orientation.
\begin{equation*}
\xymatrixrowsep{0.3pc}
\xymatrix{
1 & 2 & \cdots & n-1 \\
\bullet  & \bullet \ar[l] & \cdots \ar[l] & \bullet \ar[l]
}
\end{equation*} 

 For this quiver we have vertex set $ I = \{1, \dots, n-1 \} $ and edge sets
\begin{equation*}
\Omega = \{ 2 \to 1, 3 \to 2, \dots, n-1 \to n-2 \}  \quad \overline{\Omega} = \{ 1 \to 2, 2 \to 3, \dots, n-2 \to n-1 \}
\end{equation*}
So a $ \Lambda $-module $M$ consists of an $ I $-graded vector space $ M = \oplus_{i \in I} M_i $ with linear maps
\begin{equation*}
(i \to i+1) : M_i \rightarrow M_{i+1} \quad (i \to i-1) : M_i \rightarrow M_{i-1}
\end{equation*}
such that the preprojective relations 
\begin{align*}
(i+1 \to i)(i \to i+1) = (i-1 \to i)(i \to i-1) \  \text{ for }i = 1, \dots, n-1 
\end{align*}
are satisfied.  Here and later, we adopt the convention that $(1\to 0) : M_1 \rightarrow 0 $ and $ (n-1 \to n) : M_{n-1} \rightarrow 0 $ are 0.

If $ M $ is a $ \Lambda$-module, then it has a dimension vector $ v = (v_i)_{i \in I} \in \N^I$, where $ v_i = \dim(M_i)$.  It will be convenient to encode this as an element of $ Q_+ $ as $ \alpha_v = \sum_i v_i \alpha_i $.

\subsection{Socle of modules}
The only simple $ \Lambda $-modules are the one-dimensional modules $ S_i $, which have dimension 1 in the $i$th slot and 0 elsewhere.

If $ M $ is any $ \Lambda $-module, then the socle of $ M $ is defined to be the maximal semisimple submodule of $ M $.  The $S_i$th isotypic component of the socle of $ M $ is called the $ i$-socle of $ M $ and is denoted $ \soc_i(M)$.  

More explicitly, $ \soc(M) $ is the submodule of $ M $ whose $ i$th graded piece is
\begin{equation*}
\soc_i(M) = \{ w \in M_i : (i \to i+1)(w) = 0 \text{ and } (i \to i-1)(w) = 0  \}
\end{equation*}
All arrows act by $ 0 $ in $ \soc(M)$.

\subsection{Lusztig quiver varieties}
If $ v \in \N^I$, then we may consider the variety $ Rep_v $ of representations of $ \Lambda $ on a fixed $I$-graded vector space of dimension $ v $.  By the work of Lusztig \cite{L}, the irreducible components of $ Rep_v $ index the canonical basis for $ (U \mathfrak{n})_{\alpha_v} $, where $ U \mathfrak{n} $ denotes the universal envelopping algebra of the upper triangular subalgebra of $ \mathfrak{g} $.  In particular, the number of irreducible components of $ Rep_v $ is equal to the Kostant partition function of $ \alpha_v $.  

For each point $ x \in Rep_v $, we can consider the corresponding abstract $\Lambda$-module $ M_x $.  For $ w \in \N^I$,  we consider the variety $ Rep(w)_v $ of consisting of those points $ x \in Rep_v $ with $ \dim \soc_i (M_x) \le w_i $ for all $ i $.  Under Lusztig's construction the components of these varieties are related to the irreducible representations as follows.  Let $ \lambda = \lambda_w := \sum_i w_i \omega_i $ and $ \mu = \lambda_w - \alpha_v $ (here $ \omega_i $ are the fundamental weights).  The irreducible components of $ Rep(w)_v $ index the canonical basis for the $ \mu $ weight space of the irreducible representation $V(\lambda) $ of $SL_n$. 

\section{Modules with one-dimensional socle} \label{se:class}

\subsection{The Maya modules}
Let $ A $ be a proper subset of $ \{1, \dots, n \} $ of size $ i$, other than $ \{1, \dots, i \} $. 

The Maya module $N(A) $ has the following description.  If $ A = \{a_1 < \dots < a_i \}$, then $ N(A) $ has basis 
\begin{equation*}
w_{1, 1}, \dots, w_{a_1 -1, 1}, \dots, w_{k,k}, \dots, w_{a_k-1, k}, \dots, w_{i, i }, \dots, w_{a_i - 1,i}
\end{equation*}
where $w_{j,k} \in N(A)_j $.

We define
\begin{equation} \label{eq:Maya}
\begin{aligned} 
(j \to j-1)(w_{j,k}) &= w_{j-1, k} \\
(j \to j+1)(w_{j,k}) &= w_{j+1, k+1} 
\end{aligned}
\end{equation}

Note that $ N(A) $ has a 1-dimensional socle $ S_i$, spanned by $ w_{i,i} $.

Let us call the span of $ w_{k,k}, \dots, w_{a_k -1, k} $ the $k $th ``row'' of $ N(A) $ and let us call $ N(A)_j $ the $ j$th ``column''. So the $k$th row starts at column $ k $ and extends to column $ a_k -1 $.  This can be seen in the following picture of the module $ N(\{3, 6, 7 \})$.
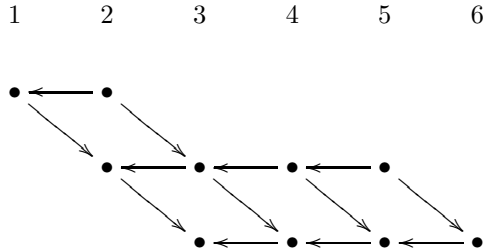
\begin{equation*}
\xymatrixrowsep{1.5pc}
\xymatrix{
1 & 2 & 3 & 4 & 5 & 6 \\
\bullet \ar[rd] & \bullet \ar[l] \ar[rd] & & & & \\
& \bullet \ar[rd] &  \bullet \ar[l] \ar[rd] &\bullet \ar[l] \ar[rd] &\bullet \ar[l] \ar[rd] & \\
& & \bullet & \bullet \ar[l] &\bullet \ar[l]  &\bullet \ar[l] 
}
\end{equation*}

\begin{Lemma}
Let $ v = \dim(N(A))$. Then $ \alpha_v = \{1, \dots, i \} - A $.
\end{Lemma}

\begin{proof}
Since we have an explicit basis for $ N(A) $ it is easy to see that
\begin{equation*}
\dim N(A)_j = | \{ r \in \{1, \dots, i \} : r \le j < a_r \} |
\end{equation*}
From this, the desired result follows immediately.
\end{proof}

\subsection{The uniqueness theorem}
We will now show that every $ \Lambda $-module with 1-dimensional socle is isomorphic to a Maya module.

We start by characterizing the dimension vectors of modules with 1-dimensional socle.  If $ v $ is a dimension vector, we will extend $ v $ by defining $ v_0 = 0 = v_n $ (this will eliminate some special cases below).

\begin{Lemma} \label{th:dimcond}
Let $M $ be a module with socle $ S_i $.  Let $ v = \dim(M) $.  Then
\begin{gather}
\label{eq:left} v_j=v_{j+1} \ \mathrm{or} \   v_j +1 = v_{j+1}, \text{ for all } j<i \\
\label{eq:right} v_{j-1}=v_j \ \mathrm{or} \  v_{j-1}= v_j +1, \text{ for all } j>i
\end{gather}
\end{Lemma}

\begin{proof}
Suppose that (\ref{eq:left}) does not hold for some $ j < i $.
Then either $v_j >v_{j+1}$ or $v_j+a=v_{j+1}$ for some $ a>1$.

Suppose that $v_j>v_{j+1}$.  Then $\dim\ker(j \to j+1)>0 $.

Consider a non-zero element $w \in \ker(j \to j+1)$. Then 
\begin{equation*}
(j+1 \to j)\circ (j \to j+1)(w)=0  \Rightarrow
(j-1 \to j)\circ (j \to j-1)(w)=0
\end{equation*}
But, $(j \to j-1)(w)\ne 0$, since $M$ has no $j$-socle. Hence
\begin{equation*}
\dim\ker(j-1 \to j)>0.
\end{equation*}
Continuing in this manner we see that $ \dim\ker(1 \to 2)>0 $.  This means that $ M $ has $1$-socle, a contradiction.

Now, suppose that $v_j+a=v_{j+1}$ for some $a>1$.  Assume $ j + 1 < i $.  In this case,
\begin{equation*}
\dim\ker(j+1 \to j) \ge a >1
\end{equation*}
Let $ w \in \ker(j+1 \to j) $, then 
\begin{equation*}
(j \to j+1)\circ (j+1 \to j)(w)=0 \Rightarrow
(j+2 \to j+1)\circ (j+1 \to j+2)(w)=0 
\end{equation*}
But $(j+1 \to j+2)(w)\neq 0$, since $M$ has no $j+1$-socle. Therefore $ (j+1 \to j+2) $ gives us an injective map from $ \ker(j+1 \to j) $ to $ \ker(j+2 \to j+1) $ and so
\begin{equation*}
\dim \ker(j+2 \to j+1) \geq a
\end{equation*}
Continuing in this manner, when we reach $ i $, we see that since the socle is one-dimensional, it must be the case that $ \dim \ker \big( (i \to i+1) |_{\ker (i \to i-1)} \big) = 1 $ and hence we find
\begin{equation*}
\dim \ker(i+1 \to i)\geq a-1>0
\end{equation*}
Again, we consider an element $w \in \ker (i+1 \to i)$. Then
\begin{equation*}
(i \to i+1) \circ (i+1 \to i) (w) = 0 \Rightarrow
(i+2 \to i+1) \circ (i+1 \to i+2) (w) = 0
\end{equation*}
Again, since $M $ does not have $i+1$-socle,
\begin{equation*}
\dim \ker(i+2 \to i+1)>0
\end{equation*}
Continuing in this manner, we see that $\dim \ker(n-1 \to n-2)>0$, which implies that $M$ has $(n-1)$-socle. This is a contradiction.

The proof of (\ref{eq:right}) follows similarly.
\end{proof}

\begin{Lemma}
Suppose that $ v \in \N^I $ satisfies the condition (\ref{eq:left}) and (\ref{eq:right}). 
 Then $ \alpha_v = \{1, \dots, i \} - A $ for some $i $ element subset of $\{1, \dots, n\} $, $ A  \ne \{1, \dots, i\}$.
\end{Lemma}

\begin{proof}
Let $\alpha_v=\sum ^{n-1}_{j=1}v_j\alpha _j$, and let $x_j$ be the $j^{th}$ coordinate of $ \alpha_v $. Then for each $ j = 1, \dots, n $, we have
\begin{align*}
x_j=1 & \iff v_j=v_{j-1}+1, \\
x_j=-1 & \iff v_j+1=v_{j-1}, \\
x_j=0 & \iff v_j=v_{j-1}.
\end{align*}
Also note that $x_j=1 \Rightarrow j \le i$ and $x_j=-1 \Rightarrow j>i$. So define 
$$ A := \{j \le i : x_j = 0 \} \cup \{j > i: x_j = 1 \}$$
and then it is easily seen that $ A $ has the desired properties.
\end{proof}

\begin{comment}
The following Lemma will also be very useful.

\begin{Lemma} \label{th:goingtosocle}
Let $M $ be a module with socle $ S_i $.  Let $ k $ be the largest index such that $ \dim(M_k) \ne 0 $.  Then there exists $ w \in M_k $ such that $ (i+1 \to i) \cdots (k \to k-1) (w) $ is a basis for the socle of $ M $.
\end{Lemma}

\begin{proof}
Consider the submodule generated by $ M_k $.  This submodule is non-zero (since $ M_k $) is non-zero and hence it contains the socle of $M $ (since the socle of $ M $ is only one dimensional).  So there exists a vector $ w \in M_k $ and an element $ \rho \in \Lambda $ such that $ \rho w $ is a basis for the socle of $ M $.  Since $ \rho \in \Lambda $ it is a sum of paths in the double quiver.  Assume for the moment that $ \rho $ is a single path.  So $ \rho $ consists of rightward and leftward steps.  Using the preprojective relations, we can assume that all of the rightward steps come before the leftward steps.  However, if there are any rightwards steps at the beginning, these will annihilate $ w $ (by the definition of $ k $).  Hence, we see that if $ \rho w $ is a basis for the socle of $M $, then it must be the case that $   (i+1 \to i) \cdots (k \to k-1) (w) $ is a basis for the socle of $M $.
\end{proof}
\end{comment}

Now we formulate and prove the uniqueness statement.   

\begin{Theorem} \label{th:uniquesoc}
Let $M $ be a module with socle $ S_i $ and dimension $ v $.  Let $ A $ be such that $ \alpha_v = \{1, \dots, i \} - A$.  Then $ M \cong N(A) $.
\end{Theorem}

This result is well-known to experts.  For example, it follows from the fact that certain Nakajima quiver varieties are 0-dimensional.  It can also be proved using the crystal structure on components of quiver varieties (due to Kashiwara-Saito \cite{KS}).   Here we prefer to give an elementary argument.

\begin{proof}
Our goal is to find a basis for $ M $ whose module structure matches the Maya module structure (\ref{eq:Maya}).  Let $ A = \{ a_1 < \dots < a_i \}$.

Let $ w_{i,i} \in M_i $ be a basis for the socle of $M$. Assume that $ a_i > i+1 $.  We claim that there exists $ w_{i+1, i} \in M_{i+1} $ such that $ (i+1 \to i)(w_{i+1, i}) = w_{i,i} $.   

Suppose that no such $ w_{i+1, i} $ exists.  From the proof of Lemma \ref{th:dimcond}, we see that $ w_{i,i} $ spans the kernel of $ (i \to i+1) $.   Hence if $ w_{i,i} $ is not in the image of $ (i+1 \to i) $, then $ (i \to i+1) \circ (i+1 \to i) $ is an isomorphism.  By the preprojective relations, this means that $ (i+2 \to i+1) \circ (i+1 \to i+2) $ is an isomorphism.  From the proof of Lemma \ref{th:dimcond}, we know that $(i+2 \to i+1) $ is injective, so both $ (i+2 \to i+1) $ and $ (i+1 \to i+2) $ are isomorphisms.  Hence $ (i+1 \to i+2) \circ (i+2 \to i+1) $ is an isomorphism.  Continuing in this fashion, we find that all $ (j \to j+1) $ are isomorphisms for $ j \ge i $ and so we see that $ v_{i+1} = \dots = v_n = 0 $.  This contradicts $ a_i > i+1 $.

By a similar argument, there exist $ w_{i+2, i}, \dots, w_{a_i - 1, i} $ such that
\begin{equation*}
w_{i,i} \xleftarrow{i+1 \to i} w_{i+1, i} \xleftarrow{i+2 \to i+1} \dots \xleftarrow{a_i-1 \to a_i -2} w_{a_i -1, i}.
\end{equation*}
Since $ (i \to i-1)(w_{i,i}) = 0 $, from the preprojective relations we find that $ (k \to k+1) (w_{k,i}) = 0 $ for all $ k $.  Thus $ w_{i, i}, \dots, w_{a_i - 1, i} $ spans a submodule which we denote by $ N $.  Note that $ N \cong N(\{1, \dots, i-1, a_i \})$.  

If $ M =N $, then we are done.  Suppose that $N \ne 0 $ and consider the quotient module $ M/N $.  Since $ \dim M/N = \dim M - \dim N $, we see that if $ v' = \dim M/N $, then $$ \alpha_{v'} = \{1, \dots, i-1\} - \{a_1, \dots, a_{i-1} \}.$$ 
We claim that $ \soc(M/N) = S_{i-1} $.  As above, there exists $ w \in M_{i-1} $ such that $ (i-1 \to i)(w) = w_{i,i} $ and as above $ (i-1 \to i-2)(w) = 0 $.  Hence $ [w] \in \soc(M/N) $.  

To see that there is no other socle, note that if $ [u] \in \soc(M/N)_j $, then $ (j \to j-1)(u) \in N $ and $ (j \to j+1)(u) \in N $.  Suppose that $ j < i-1$, then $ (j \to j+1)(u) \in N $ implies that $(j \to j+1)(u) = 0 $ which implies that $ u = 0 $ since $(j \to j+1) $ is injective (as in the proof of Lemma \ref{th:dimcond}).  Suppose that $j = i-1$, then the injectivity of $ (j \to j+1) $ forces $ u = w $.  If $ j = i $, then the $ [u] = 0 $, since the kernel of $ (i \to i-1) $ is spanned by $ w_{i,i} $.  Similarly if $ j > i $, then $ [u] = 0 $, since $(j \to j-1) $ is injective (as in the proof of Lemma \ref{th:dimcond}), so $ u $ must be a multiple of $ w_{j,i} $.

Thus, we have shown that $ \soc(M/N) = S_{i-1} $.  Thus by the induction hypothesis, we see that $ M/N \cong N(\{a_1, \dots, a_{i-1} \}) $ and we obtain a short exact sequence of $ \Lambda$-modules
\begin{equation*}
0 \rightarrow N \rightarrow M \rightarrow N(\{a_1, \dots, a_{i-1} \}) \rightarrow 0.
\end{equation*}
Let us pick a vector space splitting.  Thus combining the standard basis of $ N(\{a_1, \dots, a_{i-1}\}) $ with the above basis of $N $, we obtain a basis $ w_{k,l} $ for $ M $ with $ l = 1, \dots, i $ and $ k = l, \dots, a_l-1 $.  This module structure with respect to this basis does not match (\ref{eq:Maya}), since extra terms involving the basis for $N $ may enter into the result of applying quiver arrows to the basis elements of $ N(\{a_1, \dots, a_{i-1} \})$.  Hence we will now adjust our basis.

In particular, for each $ l = 1, \dots, i-1 $ and $ k = i+1, \dots a_l-1 $, we see that there is a scalar $ c_{k,l}$ such that
\begin{equation*}
(k \to k-1)(w_{k,l}) = w_{k-1, l} + c_{k,l}w_{k-1,i}
\end{equation*}
We may eliminate this scalar by setting $ w'_{k,l} = w_{k,l} - c_{k,l} w_{k,i}$ for these $ (k,l) $ and $ w'_{k,l} = w_{k,l} $ otherwise.  

Next, note that $ (i-1 \to i)(w_{i-1, i-1}) = 0 $ in $ N(\{a_1, \dots, a_{i-1}\}) $ and thus $ (i-1 \to i)(w'_{i-1, i-1}) = c w'_{i,i}$ in  $ M $ for some scalar $c $. Since $ M$ has no $ i-1 $ socle, $ c $ is non-zero.  Scaling all $ w'_{k,l} $ by $1/c $ (for $ l < i $), we may assume that $ c = 1$.  It then follows from the preprojective relations that  \begin{equation*}
 (k \to k+1)(w'_{k,i-1}) = w'_{k+1, i} \text{ for all } k = i, \dots, a_{i-1} - 1.
 \end{equation*}

Now consider some $ w'_{k,l} $ for $ l < i-1$ and $ k \ge i-1 $.  Then
\begin{equation*}
(k \to k+1)(w'_{k,l}) = w'_{k+1, l+1} + c_l w'_{k+1, i}
\end{equation*}
for some scalar $ c_l $.  By the preprojective relations $ c_l $ depends only on $ l $.  Then we make the adjustment $ w''_{k,l} = w'_{k,l} - c_l w'_{k, i-1} $ for all $k = i-1, \dots, a_{l-1} - 1 $ and $ w''_{k,l} = w'_{k,l} $ for all other $ (k,l) $.

After all these adjustments, we see that $ w''_{k,l} $ satisfy the Maya module structure $ (\ref{eq:Maya})$.  Thus $ M \cong N(A) $ as desired.
\end{proof}

\subsection{Computation of hom spaces}
Now we compute the space of homomorphisms between Maya modules.

\begin{Theorem} \label{th:homAB}
Let $A$, $B$ be $i$, $j$ element subsets respectively.  Then we have
\begin{align*}
\dim \Hom(N(A),N(B)) = &\text{ \# of $ r \in \{1, \dots, i \} $, such that } r \le j < a_r, \\
& \quad \text{ and } a_{r-l} \le b_{j-l} \text{ for } l = 0, \dots, r-1
\end{align*}
where $ A = \{ a_1 < \dots < a_i \} $ and $ B = \{ b_1 < \dots < b_j \}$.
\end{Theorem}

\begin{proof}
Let 
\begin{equation*}
R := \bigl\{ r \in \{1, \dots, i \}: r \le j < a_r, \text{ and } a_{r-l} \le b_{j-l} \text{ for } l = 0, \dots, r-1 \bigr\} 
\end{equation*}

We construct a map $\varphi: R \to \mathrm{Hom}(N(A),N(B))$, and then show that it gives a bijection between $R$ and a basis for $\mathrm{Hom}(N(A),N(B))$. This will yield the desired result.

For simplicity of notation, we will use $w_{k,l} $ for the basis for $ N(A) $ and $ w_{k,l}' $ for the basis for $ N(B) $.

For each $ r \in R$, let us define $ \varphi(r) = \phi_r $ to be the homomorphism which takes the $ r $th row of $N(A)$ to the bottom row of $ N(B) $ and then extended to higher rows in the obvious way.  More explicitly, we define $ \phi_r $ by 
\begin{equation*}
\phi_r(w_{k, r-l}) = \begin{cases}
&w'_{k, j-l}, \ \text{ if } l \ge 0, \text{ and } k \ge j-l \\
&0, \ \text{ otherwise}
\end{cases}
\end{equation*}
Such a $ w'_{k,j-l} $ will always exist since $ j-l \le k < a_{r-l} \le b_{j-l} $.  A simple check using the structure of Maya modules (\ref{eq:Maya}) shows that $ \phi_r $ is a homomorphism.

Now, suppose that $ \psi $ in any element of $ \Hom(N(A), N(B)) $.  Since we have explicit bases for $ N(A) $ and $N(B) $ we may consider the matrix coeffients involving $ w'_{j,j}$, the generator of the socle of $N(B) $.  $\psi $ takes $ N(A)_j $ to $ N(B)_j $ so for each $ r \in \{1, \dots, i \} $ such that $ r \le j < a_r $, we get a matrix coefficient $ s_r $, such that 
\begin{equation*}
\psi(w_{j,r}) = s_r w'_{j,j} + \cdots.
\end{equation*}

Note that if all the $ s_r $ are zero, then $ \psi = 0 $.  This is because every submodule of $ N(B) $ must contain $ w'_{j,j}$ (since $ w'_{j,j}$ spans the socle of $ N(B) $) and so any non-zero homomorphism from $ N(A) $ to $ N(B) $ must hit $ w'_{j,j} $.  Thus, the collection $ s_r $ completely determines $ \psi$.  

Also note that if $ r \notin R $, then $ a_{r-l} > b_{j-l} $ for some $ l $.  This means that we can find some non-zero $ w \in N(A) $ and $ p \in \Lambda $ such that $ pw = w_{j,r} $ but $\psi(w) = 0 $ (in fact we can choose $ w = w_{a_{r-l} - 1, r-l} $).  Hence for $ r \notin R $, we see that $ s_r = 0 $.

Combining these observations, we see that $ \psi = \sum_{r \in R} s_r \phi_r $.  Thus the $ \phi_r $ span $\Hom(N(A), N(B))$.  These $\phi_r $ are linearly independent since $ \phi_r $ vanishes on $ a_{j, r'} $ for $ r > r'$.  Thus the $ \phi_r $ form a basis for $ \Hom(N(A), N(B)) $ as desired.
\end{proof}

\subsection{Connection with MV polytopes}
We now make the connection between Theorem \ref{th:homAB} and MV polytopes.

For each subset $ B$ of $ \{1, \dots, n\}$ of size $ i $, we may consider the truncated permutahedron $ P(B) $ which is defined as 
\begin{equation*}
P(B) := conv(\{1_{C} - 1_{\{1, \dots, j\}} : \text{ $C $ is a subset of $\{1, \dots, n\}$ of size $ j $ and $ C \le B $ }\})
\end{equation*}

These polytopes $ P(B) $ are relevant since Naito-Sagaki \cite{NS} have shown that these are the MV polytopes associated to the vertices of the crystal corresponding to the minuscule representation $ \Lambda^i \mathbb{C}^n$.  These vertices are precisely labelled by subsets $ B $ of size $ i $.

\begin{Corollary}
For each subset $ B $ of $\{1, \dots, n\}$, the max value of $ \langle 1_A, \rangle $ on the polytope $ P(B) $ is given by $\dim \Hom(N(A),N(B))$.
\end{Corollary}

\begin{proof}
Assume for simiplicity that $ i \le j $.  A similar proof holds in the $ i > j $ case.

By Theorem \ref{th:homAB}, $\dim \Hom (N(A), N(B)) = r - s $ where $ r $ is the maximal element of $ \{1, \dots, i \} $ such that $ a_{r - l} \le b_{j-l} $ for $ j = 0, \dots, r-1 $ and $ s = | \{1, \dots, j \} \cap A | $.

Now, we claim that $ r = \max_{C \le B} | C \cap A | $.  First note that if we choose $ C $ to be the smallest possible $j $ element subset of $ \{1, \dots, n\} $ such that $  \{a_1, \dots, a_r \} \subset C $, then $ C \le B $ and $ |C \cap A| \ge r $.
On the other hand, for any $ C \le B $, we claim that $ | C \cap A | \le r $.  To see why this is the case, note that by the definition of $ r$, not all the inequalities 
\begin{equation} \label{eq:someineq}
a_{r+1} \le b_j, a_r \le b_{j-1}, \dots, a_1 \le b_{j-r} .
\end{equation} 
can hold.  So now suppose that $ C \le B $ and $ C \cap A $ contains at least $ r+1 $ elements.  Let us choose $ r+1 $ of these elements and order them  $ a_{i_1} < \dots < a_{i_{r+1}} $.  Then since $ C \le B $, we find that
\begin{equation*}
a_{i_{r+1}} \le b_j, \dots, a_{i_1} \le b_{j-r} .
\end{equation*}
But since $ a_{i_l} \ge a_l $ for all $ l $, this implies that all the inequalities (\ref{eq:someineq}) hold --- a contradiction.  Hence we conclude that $r = \max_{C \le B} | C \cap A | $.

Thus 
\begin{align*}
\dim \Hom (N(A), N(B)) = r - s &= \max_{C \le B} | C \cap A | - |  \{1, \dots, j \} \cap A | \\
&= \max_{C \le B} \langle 1_A, 1_C - 1_{\{1, \dots, j\}} \rangle
\end{align*}
as desired.
\end{proof}

\section{Description of irreducible components} \label{se:tableaux}
 
\subsection{Savage's description of the components}
Alistair Savage has given a description of the components of $ Rep(w)_v $ in terms of tableaux.  We would like to reformulate his description in terms of Hom spaces.

Let $ Tab_\mu(\lambda)$ denote the set of semistandard Young tableaux (SSYT) of shape $ \lambda $ and content $ \mu$.  If $ X $ is a box in a SSYT $ T $, then we will write $ r(X) $ for the row of $X $ and $ c(X) $ for the content of $ X$.

For each $ T \in Tab(\lambda)_\mu $, Savage has identified a component $ C_T $ of $ Rep(w)_v $. 

Let $ T \in Tab(\lambda)_\mu $.  A $ \Lambda$-module is said to be of type $ T$ if there exists a basis for $ M $ with the following properties.  For each box $ X $  in $ T $, there are vectors $$ w_{r(X)}^X, \dots, w_{c(X) -1}^X \in M$$ and the collection of all these vectors (over all boxes $ X$) forms a basis for $M $.  Moreover, we have 
\begin{equation} \label{eq:Mstruct}
(j \to j-1) (w_j^X) = w_{j-1}^X, \quad (j \to j+1)(w_j^X) = \sum_Y d_Y^X w_{j+1}^Y
\end{equation}
for some scalars $ d_Y^X $, where the sum varies over all those boxes $ Y $ such that $ r(Y) < r(X) \le c(Y) < c(X) $.

Let $ C_T = \overline{\{ x \in Rep(w)_v : M_x \text{ is of type } T \}}$ denote the closure of the locus of those modules of type $ T$.

\begin{Theorem}[{\cite[Section 5]{Sav}}] \label{th:Sav} 
$C_T $ is a component of $ Rep(w)_v $ and this provides a bijection between the components of $ Rep(w)_v $ and $ Tab(\lambda)_\mu$.
\end{Theorem}

\subsection{Description of components by Hom spaces}
We would like to reformulate Savage's description.  The key will be the following generalization of Theorem \ref{th:homAB}.  A connnected subset of $ \{1, \dots, n\} $ is one of the form $ \{ t-i+1, t-i+2, \dots, t \} $.

\begin{Theorem}\label{th:HomT}
Let $ M $ be a module of type $ T $ and let $ A = \{ t-i+1, \dots, t \}$ be a connected subset of $ \{1, \dots, n \}$.  Then 
\begin{equation*}
\dim \Hom(M, N(A)) = \text{\# of boxes $X $ in $ T $, such that } r(X) \le i <  c(X) \le t.
\end{equation*}
\end{Theorem}

\begin{proof}
The idea is similar to the proof of Theorem \ref{th:homAB}.

To each box $ X $ of $ T $ in the above set, we can define a homomorphism $ \phi_X : M \rightarrow N(A) $ by taking the row indexed by $ X $ to the bottom row of $ N(A) $.  We then extend to all of $ M $.

More explicitly, we define
\begin{equation*}
\phi_X(w_k^X) = w'_{k,i} 
\end{equation*}
for $ k \ge i $ (note that such $ w'_{k,i} $ exists since $ i \le k < c(X) \le t = a_i$).  We define $ \phi_X(w_k^X) = 0 $ for $ k < i $.  We also define $ \phi_X(w_k^Y) = 0 $ for all $ Y \ne X $ with $ r(Y) \ge r(X)$ or $ c(Y) \ge c(X) $.

Now we proceed to define $ \phi_X(w_k^Y) $ for those boxes $ Y $ with $ r(Y) < r(X) $ and $ c(Y) < c(X) $.  We do so by an inductive procedure on $ r(Y) $.  Suppose that $ Y $ is a box with $ r(Y) = r(X) - 1$ and $ c(Y) < c(X) $.  Then we define
\begin{equation*}
\phi_X(w_k^Y) = d_Y^X w'_{k, i-1}.
\end{equation*}
Note that such $ w'_{k, i-1} $ exists since $ k < c(Y) < c(X) \le t $ and so $ k < t-1 = a_{i-1} $.  

Next, suppose that $ Y $ is a box with $ r(Y) = r(X) - 2 $ and $ c(Y) < c(X) $.  Then we define
\begin{equation*}
\phi_X(w_k^Y) = d_Y^X w'_{k, i-1} + \sum_Z d_Y^Z d_Z^X w'_{k, i-2},
\end{equation*}
where the sum ranges over all those boxes $Z $ such that $ r(Z) = r(X) - 1 $ and $ c(Y) < c(Z) < c(X) $.

Continuing in this fashion, we define $ \phi_X $ on all of $M $.  The structure of the module as given in (\ref{eq:Mstruct}) ensures that $ \phi_X $ is a $ \Lambda$-module homomorphism.

The fact that these $ \phi_X $ form a basis for $\Hom(M, N(A)) $ follows along the same lines as in the proof of Theorem \ref{th:homAB}.
\end{proof}

We now combine this result and Savage's theorem.  For each $ A = \{ t-i+1, \dots, t \}$, we define a constructible function
\begin{align*}
 f_A : Rep(w)_v &\rightarrow \mathbb{N} \\
  x &\mapsto \dim(\Hom(M_x,N(A)) 
\end{align*}
Since this is a constructible function it takes a constant value on a constructible dense subset of each component of $ Rep(w)_v $. For each component $ Z \subset Rep(w)_v $, let $ f_A(Z) $ denote this constant value.

Also, for each $ T \in Tab(\lambda)_\mu $, let $ g_A(T) $ denote the number of boxes $ X $ in $ T$ such that $ r(X) \le i < c(X) \le t$.  Note that the collection $ \{ g_A(T) \} $ (where $ A $ ranges over all connected subsets) determines $ T $.

\begin{Theorem}
For each component $ Z \subset Rep(w)_v $, there exists a tableau $ T \in Tab(\lambda)_\mu $ such that $ f_A(Z) = g_A(T) $ for all connected subsets $ A \subset \{1, \dots, n \} $.  This provides a bijection between the components of $ Rep(w)_v $ and the SSYT of shape $ \lambda $ and filling $ \mu $.
\end{Theorem}

\begin{proof}
Theorem \ref{th:HomT} shows that if $ M $ is of type $T$, then $ f_A(M) = g_A(T) $.  Theorem \ref{th:Sav} shows that for each component $ Z $, there exists a unique tableau $ T $ such that there is a dense subset of $ Z $ consisting of modules of type $ T $.  Combining these two results, we obtain the desired result.
\end{proof}

\end{document}